\documentclass[11pt]{amsproc}
\usepackage{tikz-cd} 
\usepackage{amsmath}
\usepackage{amsthm}
\usepackage{amsfonts}
\usepackage[normalem]{ulem}

\newtheorem{thm}{Theorem}

\newtheorem{lem}[thm]{Lemma}

\newtheorem{cor}[thm]{Corollary}
\newtheorem{exmp}[thm]{Example}
\newtheorem{rem}[thm]{Remark}

\newcommand{\Pset}{\mathcal{P}}
\newcommand{\D}{\mathcal{D}}

\newcommand{\U}{\mathcal{U}}
\newcommand{\Nil}{\mathcal{N}}
\newcommand{\cO}{\mathcal{O}}

\newcommand{\FF}{\mathbb{F}}

\DeclareMathOperator{\im}{Im}
\DeclareMathOperator{\rank}{rank}

\usetikzlibrary{tikzmark}
\newcounter{pairctr}

\usepackage{amssymb}
\usepackage{natbib}
\usepackage{graphicx}

\title{Jordan types commuting with a hook partition}

\author{Leila Khatami}
\address{Department of Mathematics, Union College, 807 Union St, Schenectady, NY 12308,
USA} 
\email{khatamil@union.edu}

\author{Toma\v z Ko\v sir}
\address{Faculty of Mathematics and Physics, University of Ljubljana, Jadranska 19, 1000 Ljubljana, Slovenia, and Institute of Mathematics, Physics and Mechanics, Jadranska 19, 1000 Ljubljana, Slovenia} 
\email{tomaz.kosir@fmf.uni-lj.si}

\date{\today}

\thanks{The work of the first author was partially supported by Faculty Research Fund from Union College. The second author acknowledges financial support from the Slovenian Research and Innovation Agency (by research program No. P1-0448 until December 31, 2025, and by research program No. P1-0222 and research project No. J1-50002 after January 1, 2026). }

\begin{document}

\begin{abstract}

 We give a complete classification of the Jordan types occurring in the nilpotent commutator of a nilpotent matrix whose Jordan type is a hook partition. As a consequence, we also show that two partitions with the same generic commuting Jordan type need not commute with each other.
\end{abstract}

\maketitle

\section{Introduction}

The question of which partitions of a given natural number can occur as the Jordan types of two commuting nilpotent matrices has attracted substantial interest in recent decades. We refer the readers to surveys \cite{Kha-S,Kos-S} for more information. Since the nilpotent commutator of a nilpotent matrix with Jordan type $P$ is an irreducible variety (as proved by Basili in \cite[Lem. 2.3]{Bas-2}), its intersection with one of the nilpotent orbits is a dense subset. The partition corresponding to the dense orbit studied by Panyushev in \cite{Pan} for the first time, is denoted by $\mathcal{D}(P)$. It is a super-distinct or a Rogers-Ramanujan partition, that is, it is a partition with parts differing by at least $2$ \cite[Thm. 1.12]{BI}. The largest and the smallest part of $\D(P)$ were given by Oblak \cite{Obl-1} and Khatami \cite{Kha-2}, respectively. The map $\mathcal{D}$ is idempotent \cite[Thm. 6]{KO}. The authors of \cite[Conj. 4.11]{IKvSZ} formulated the Box Conjecture on the structure of the inverse image $\mathcal{D}^{-1}(Q)$ of a Rogers-Ramanujan partition $Q$. The conjecture was recently proved in \cite[Thm. 1]{IKM}. 

It is known that every nilpotent commutator intersects all nilpotent orbits of Jordan type $P=(2^k,1^l)$, where $k$ and $l$ are integers (see \cite[Thm. 2.4]{Obl-3} or \cite[Thm. 4.11]{BW}). Such partitions are called universally commuting. On the other hand, if the Jordan type of a matrix $B$ is a Rogers-Ramanujan partition then no other nilpotent orbit corresponding to a Rogers-Ramanujan partition intersects the nilpotent commutator of $B$ \cite[Thm. 5.3]{Kos-S}. 

The authors in \cite[Thm. 3]{BDKOS} described the nilpotent orbits that intersect the commutator of a nilpotent matrix $A$ of Jordan type $(n,n)$. In this paper, we consider a nilpotent matrix whose Jordan type is a hook partition $(n,1^m)$, where $n\ge 3$ and $m\ge 1$, and describe all the nilpotent orbits that intersect its nilpotent commutator. Note that then $\D(n,1^m)$ is a Rogers-Ramanujan partition with two parts. Let us also mention that the special case of our result for $m=1$ was proved by Britnell and Wildon in \cite[Prop. 4.9]{BW}. 

The paper is organized as follows. In the next section we introduce the notation and our setting, and in section \ref{hook} we present and prove our main results. We describe all partitions that commute with a hook partition $(n,1^m)$. In particular, we show that a necessary condition for a partition to commute with $(n,1^m)$ is to contain a subpartition that is an almost rectangular partition of either $n$, $n-1$ or $n-2$. In section \ref{consequences}, we discuss consequences of the main result and provide examples. Among these are examples that show that, in general, the hook partitions in the inverse image $\D^{-1}(Q)$ need not commute with every other partition in $\D^{-1}(Q)$.

\section{Preliminaries}

We assume that $\FF$ is an infinite field of characteristic $0$ and we denote by $\operatorname{Mat}_N(\FF)$ the set of all $N\times N$ matrices over $\FF$. 

We denote by $\Pset(N)$ the set of all partitions of a natural number $N$. A partition $P\in\Pset(N)$ is written as $P=(\lambda_1,\lambda_2,\ldots,\lambda_k)$, where $\lambda_i$ are arranged in nonincreasing order,  $\lambda_k> 0$ and $\sum_{i=1}^k \lambda_i=N$ . 

Suppose $B$ is a nilpotent matrix and $P=P_B=(\lambda_1,\lambda_2,\ldots,\lambda_k)\in\Pset(N)$ is its Jordan type, i.e., $\lambda_1\ge \lambda_2\ge\ldots\ge \lambda_k>0$ are sizes of its Jordan blocks. The nilpotent commutator $\Nil_B$ of $B$ is defined by 
$$\Nil_B=\{A\in \operatorname{Mat} _N(\FF):\,BA=AB,A^N=0\}.$$ 

We are interested in partitions $Q\in\Pset(N)$ such that the nilpotent orbit $\cO_{Q}$, that is the set of all nilpotent matrices of Jordan type $Q$, intersects $\Nil_B$ nontrivially, i.e., such that $\cO_{Q}\cap\Nil_B\neq\emptyset$. 

In connection with commuting Jordan types, two classes of partitions play an important role. A partition $Q=(\mu_1,\mu_2,\ldots,\mu_k)$ is called \emph{super-distinct or Rogers-Ramanujan} if its parts differ by at least $2$: $\mu_i-\mu_{i+1}\ge 2$ for $i=1,2,\ldots,k-1$.

On the opposite side, a partition $Q=(\mu_1,\mu_2,\ldots,\mu_k)$ is called \emph{almost rectangular}  if its largest and smallest parts differ by at most $1$: $\mu_1-\mu_{k}\le 1$. Note that if $J_N$ is the $N\times N$ Jordan nilpotent matrix with one Jordan block of size $N$, then the Jordan types of its powers $J_N^k$, $k=1,2,\ldots,N,$ are exactly all almost-rectangular partitions. We denote by $[N]^k$ the Jordan type of $J_N^k$.

\section{Commuting with a hook partition}\label{hook}

In this section, we characterize all possible Jordan types in the nilpotent commutator of a hook partition $\Gamma_{n,m}=(n,1^m)$. We write $N=n+m$. Let $B$ be the nilpotent Jordan matrix with Jordan type $\Gamma_{n,m}$, and let $\U_{B}$ be the maximal nilpotent subalgebra of the nilpotent commutator $\Nil_B$ as in \cite{BaIK2010}. Then $\U_B$ intersects every nilpotent orbit that intersects $\Nil_B$ \cite[Lem. 2.2]{BaIK2010}, and consists of all matrices of the following form in $\operatorname{Mat}_{N}(\mathbb{F})$, the space of all $N$ by  $N$ matrices over the field $\mathbb{F}$.

\begin{equation}\label{generic in U}
A=\left(\begin{array}{ccccc|ccccc}
    0 & h_1 & h_2 & \cdots & h_{n-1} & u_1 & u_2 & \cdots & u_{m-1} & u_{m}\\
    0 & 0   & h_1 & \cdots & h_{n-2} & 0   & 0   & \cdots & 0        & 0\\
    \vdots &     & \ddots &        & \vdots & \vdots & \vdots &        &        & \vdots\\
    0 & 0   & 0   & \cdots & h_1     & 0   & 0   & \cdots & 0        & 0\\ 
     0 & 0   & 0   & \cdots & 0     & 0   & 0   & \cdots & 0        & 0\\
    \hline
    0 & 0   & 0   & \cdots & d_1     & 0   & 0   & 0 & \cdots       & 0\\
    0 & 0   & 0   & \cdots & d_2     & v_{12}   &0 & 0& \cdots      &0\\
    0 & 0   & 0   & \cdots & d_3     & v_{13}& v_{23}& 0 & \ddots  & 0\\
    \vdots &     &        & \cdots & \vdots & \vdots &       & \ddots & \ddots & \vdots\\
    0 & 0   & 0   & \cdots & d_m     & v_{1m} & v_{2m} & \cdots & v_{m-1,m} & 0\\
\end{array}
\right)
\end{equation}

Let $\{e_1, \ldots, e_n\}$ be the standard basis of $\mathbb{F}^n$. Then we can write 
\begin{equation}\label{blocks of generic in U}
A=\left( 
\begin{array}{c|c}
  H   & e_1 u^T \\
  \hline
  d e_n^T   & V
\end{array}
\right),
\end{equation}
where $H=\sum\limits_{i=1}^{n-1}h_iJ_n^i$, vectors $u$ and $d$ are both in $\mathbb{F}^m$, and $V$ is a strictly lower triangular matrix in $\operatorname{Mat}_{m}(\mathbb{F})$.

\begin{lem}\label{commute with hook}
    The hook partition $\Gamma_{n,m}=\left(n,1^m\right)$ commutes with any partition $Q$ of $m+n$ having one of the following forms:
    \begin{enumerate}
        \item[(a)] $Q$ is the partition obtained from $\left(\mu, [n]^k\right)$, where $\mu$ is an arbitrary partition of $m$ and $1\leq k\leq n$.
        \item[(b)] $Q$ is the partition obtained from $\left(\mu,[n-1]^k\right)$, where $\mu=(\mu_1,\ldots, \mu_s)$ is a partition of $m+1$, and $\frac{n-1}{\mu_1}\leq k \leq n-1$.
        \item[(c)] $Q$ is the partition obtained from $\left(\mu, [n-2]^k\right)$, where either $\mu=(m+2)$ and $\frac{n-1}{m+1}< k \leq n-2$, or $\mu=(\mu_1, \mu_2, \ldots, \mu_s)$ is a partition of $m+2$ such that $\mu_2\geq 2$, and $\frac{n-1}{\mu_2}< k \leq n-2$.
        \end{enumerate}
\end{lem}

\begin{proof}
Throughout this proof, we use the notations established in (\ref{generic in U}) and  (\ref{blocks of generic in U}) above.
\\

\noindent{\bf Part (a).}  Let $\mu$ be a partition of $m$, and assume that $1\leq k\leq n$. Let $Q$ be the partition obtained from $\left(\mu, [n]^k\right)$. In $\U_B$, we get Jordan type $Q$ by choosing $H=J_n^k$, $V=J_\mu^T$, and $u=d=0$ in (\ref{blocks of generic in U}) above.

\noindent{\bf Part (b).} Let $\mu=(\mu_1, \mu_2, \ldots, \mu_s)$ be a partition of $m+1$, and assume that $\frac{n-1}{\mu_1}\leq k\leq n-1$. Let $Q$ be the partition obtained from $\left(\mu, [n-1]^k\right)$.

If $\mu=\left(1^{m+1}\right)$, then $Q$ is the partition $\left(1^{m+n}\right)$, which corresponds to the zero matrix in $\U_B$.  For the rest of the proof of this part we assume that $2\leq \mu_1\leq m+1$.

Let $\mu'=(\mu_1-1,\mu_2, \ldots, \mu_s)$. Using the expression (\ref{blocks of generic in U}) for an element $A\in \U_B$, set $H=J_n^k$, $V=J_{\mu'}^T$, $u=\delta e'_{\mu_1-1}$ and $d=e_1'$, where $\{e_1', \ldots, e'_m\}$ is the standard basis of $\mathbb{F}^m$, and $\delta\in \mathbb{F}$. We reorganize the block decomposition of $A$ to get

\[A=\left( 
\begin{array}{c|c}
  J_n^k  & \delta e_1\left(e'_{\mu_1-1}\right)^T \\
  \hline
  e_1'e_n^T   & J_{\mu'}^T
\end{array}
\right)= 
\left( 
\begin{array}{c|c}
  J_{n-1}^k & f_{n-k}\left(f'_1\right)^T +\delta f_1\left(f'_{\mu_1}\right)^T\\
 \hline
  0 & J_\mu^T
\end{array}
\right),\]
where $\{f_1, \ldots, f_{n-1}\}$ and $\{f'_1, \ldots, f'_{m+1}\}$ are the standard bases of $\mathbb{F}^{n-1}$ and $\mathbb{F}^{m+1}$, respectively. 

Thus for a positive integer $l$, we have 

\[A^l=\left(
\begin{array}{c|c}
    J_{n-1}^{kl} & J_{n-1}^{k(l-1)}f_{n-k}(f'_1)^T+\delta f_1\left(f'_{\mu_1}\right)^T (J_\mu^{l-1})^{T}\\
    \hline
    0 & \left(J^l_{\mu }\right)^T
\end{array}
\right).
\]

\noindent {\bf Case 1.} Assume that $\frac{n}{\mu_1}\leq k$. In this case, let $\delta=0$. 

For  $1\leq l<\frac{n}{k}$, since $l<\mu_1$, the block $  \left(J^T_{\mu }\right)^l$ is non-zero and has a row equal to $(f'_1)^T$. Therefore, for such $l$, we have 
\[\rank(A^l)=\rank\left(
\begin{array}{c|c}
    J_{n-1}^{kl} &0 \\
    \hline
    0 & \left(J^l_{\mu }\right)^T
\end{array}
\right). \]

On the other hand, for $l\geq \frac{n}{k}$ we have 
\[A^l=\left(
\begin{array}{c|c}
   0 &0 \\
    \hline
    0 & \left(J^l_{\mu }\right)^T
\end{array}
\right), \]

Therefore, in this case $A$ has the desired Jordan type.

\noindent {\bf Case 2.} Now assume that $\frac{n-1}{\mu_1}\leq k <\frac{n}{\mu_1}$. This is possible only if $n-1$ is divisible by $\mu_1$ and so $n-1=k\mu_1$. Let $\delta=-1$. 

For $1\leq l \leq \mu_1-1$, we have $kl<n-1$, so that the block $J_{n-1}^{kl}$ has a column equal to $f_1$ and the block $\left(J^T_{\mu }\right)^l$ has a row equal to $(f'_1)^T$. Thus, for such $l$, we have 

\[\rank(A^l)=\rank\left(
\begin{array}{c|c}
    J_{n-1}^{kl} &0 \\
    \hline
    0 & \left(J^l_{\mu }\right)^T
\end{array}
\right). \]

For $l=\mu_1$, we have 

\[A^{\mu_1}=\left(
\begin{array}{c|c}
    0 & f_{1}(f'_1)^T- f_1\left(f'_{1}\right)^T \\
    \hline
    0 & 0
\end{array}
\right)=0=\left(
\begin{array}{c|c}
    J_{n-1}^{k\mu_1} & 0 \\
    \hline
    0 & \left(J^{\mu_1}_{\mu }\right)^{T}
\end{array}
\right), 
\]
and hence $A^l=0$ also for $l>\mu_1$.

Therefore, in this case $A$ has the desired Jordan type as well.

\noindent{\bf Part (c).} Assume that $\mu=(\mu_1,\mu_2, \ldots, \mu_s)$ is a partition of $m+2$ such that either $\mu_2=0$ or  $\mu_2\geq 2$. 

If $\mu_2=0$, let $\mu'=(m)$, and let $k$ be such that $\frac{n-1}{m+1}< k \leq n-2$. If $\mu_2\geq 2$, let $\mu'=(\mu_1-1, \mu_2-1, \mu_3, \ldots, \mu_s)$ and  $k$ be such that $\frac{n-1}{\mu_2}< k \leq n-2$.

We first note that if $\mu_1\leq 2$, then in fact $\mu_1=2$ and all parts of the partition obtained from $\left(\mu, [n-2]^k\right)$ are at most 2. We already know that such partition commutes with any partition (see \cite[Thm. 2.4]{Obl-3}). So for the rest of this proof we assume that $\mu_1\geq 3$. 

Let $A\in \U_B$, written as in (\ref{blocks of generic in U}), be such that  $H=J_n^k$, $V=\left(J_{\mu'}\right)^T$, $u=e'_{\mu_1+\mu_2-2}$, and $d=e_1'$. 
In other words, 
 \[A=\left( 
\begin{array}{c|c}
  J_n^k  & e_1\left(e'_{\mu_1+\mu_2-2}\right)^T \\
  \hline
  e_1'e_n^T   & J_{\mu'}^T
\end{array}
\right).\]

In the case $\mu=(m+2)$, for simplicity, we also write $\mu_1+\mu_2$ to mean $m+2$. 

Let $\bar{A}=P^{-1}AP$ where $P$ is the base change matrix permuting the standard basis of $\mathbb{F}^{n+m}$ by cycle  $(1\to 2\, \to \ldots \to n+\mu_1+\mu_2-2)$ and leaves other vectors fixed (if any). We can then reorganize the block decomposition of $\bar{A}$ to write 
\[\bar{A}= 
\left( 
\begin{array}{c|c}
  J_{n-2}^k & f_{n-k-1}\left(f'_1\right)^T \\
 \hline
  f'_{\mu_1+\mu_2}f_k^T & J_\mu^T
\end{array}
\right).\]

Here, $\{f_1, \ldots, f_{n-2}\}$ and $\{f_1', \ldots,f'_{m+2}\}$ are the standard bases of $\mathbb{F}^{n-2}$ and $\mathbb{F}^{m+2}$, respectively. 

Using the notation $\delta_{i,j}$ to denote the Kronecker delta, for a positive integer $l$, we have 

\[\begin{array}{ll}
\bar{A}^l&= 
\left( 
\begin{array}{c|c}
  J_{n-2}^{kl} & J_{n-2}^{k(l-1)}f_{n-k-1}\left(f'_1\right)^T \\
 \hline
  f'_{\mu_1+\mu_2}f_k^TJ_{n-2}^{k(l-1)}& \left(J^l_\mu\right)^T +\delta_{kl,n-1}f'_{\mu_1+\mu_2}(f'_1)^T
\end{array}
\right)
\\ \\
&=
\left( 
\begin{array}{c|c}
  J_{n-2}^{kl} &f_{n-kl-1}\left(f'_1\right)^T \\
 \hline
  f'_{\mu_1+\mu_2}f_{kl}^T& \left(J^l_\mu\right)^T +\delta_{kl,n-1}f'_{\mu_1+\mu_2}(f'_1)^T
\end{array}
\right).
\end{array}\]

We will show that $A$ has the desired Jordan type by showing that for every positive integer $l$, 
\[\rank\bar{A}^l=\rank \left( 
\begin{array}{c|c}
  J_{n-2}^{kl} & 0 \\
 \hline
  0& \left(J^l_\mu\right)^T
\end{array}
\right).\]

\noindent {\bf Case I.} Assume that $\mu_2\geq 2$. If $l$ is such that $1\leq l <\mu_2$, then the $(\mu_1+\mu_2-l)$-th column of $\left(J^l_\mu\right)^T$ is equal to $f'_{\mu_1+\mu_2}$. Therefore, for such $l$, we can use an elementary column operation to remove the non-zero entry in the block $\bar{A}^l_{21}$, which will be in row $\mu_1+\mu_2$ and column $kl$ of the block, as well as the extra term in the block $\bar{A}^l_{22}$, which will be in row $\mu_1+\mu_2$ and column $1$ of the block. Moreover, the $(l+1)$-st row of $\left(J^l_\mu\right)^T$ is equal to $(f'_1)^T$ and therefore a nonzero entry in the block $\bar{A}^l_{12}$, which will be in row $n-kl-1$ and column $1$ of that block, can be eliminated using an elementary row operation. 

On the other hand, for $l\geq \mu_2$, since $kl> n-1$, both blocks $\bar{A}^l_{12}$ and $\bar{A}^l_{21}$, as well as the extra term in the block $\bar{A}^l_{22}$, are 0.
\\

\noindent {\bf Case II.} Now, assume that $\mu=(m+2)$ and that $\frac{n-1}{m+1}< k \leq n-2$. We note that in this case, as long as $l\leq m$, the $(l+1)$-st row of $\left(J^l_\mu\right)^T$ is $(f'_1)^T$, and its $(m+2-l)$-th column is $f'_{m+2}$. Since $l+1\leq m+1$ and $m+2-l\geq 2$, a nonzero entry in $\bar{A}^l_{12}$ or $\bar{A}^l_{21}$, or an additional term in $\bar{A}^l_{22}$ can be eliminated through appropriate elementary column and row operations.

On the other hand, for $l\geq m+1$, we have $lk> n-1$ and therefore both $\bar{A}^l_{12}$ and $\bar{A}^l_{21}$, as well as the extra term in $\bar{A}^l_{22}$ are 0.

Therefore, in both cases, for every positive integer $l$, we have the desired rank conditions, implying that $A$ has Jordan type $([n-2]^k, \mu)$.
\end{proof}

\begin{rem}
In part (c) of Lemma~(\ref{commute with hook}), we made the lower bounds for $k$ sharp, but partitions that come from the inclusive lower bounds for $k$ are also among the commuting partitions covered in earlier parts of Lemma \ref{commute with hook}. If $\mu=(m+2)$ and $k=\frac{n-1}{m+1}$, then $([n-2]^k, m+2)=(m,(m+1)^{k-1},m+2)=([n]^k, m)$, which is already covered in Part (a). On the other hand, if $\mu=(\mu_1, \mu_2, \ldots, \mu_s)$ with $\mu_2\geq 2$, for $k=\frac{n-1}{\mu_2}$, we have $([n-2]^k, \mu)=(\mu_2^{k-1}, \mu_2-1, \mu)=([n-1]^k, \mu_1,\mu_2-1, \ldots, \mu_s)$, which is already covered in case (b) because $\frac{n-1}{\mu_1}\leq \frac{n-1}{\mu_2}\leq k$.
\end{rem}

\begin{thm}\label{all that commute with hook}
    Partition $\Gamma_{n,m} =\left(n,1^m\right)$ commutes with a partition $Q$ of $m+n$ if and only if $Q$ is one of the following partitions:
    \begin{enumerate}
        \item[(a)] $Q$ is the partition obtained from $\left(\mu, [n]^k\right)$, where $\mu$ is an arbitrary partition of $m$ and $1\leq k\leq n$.
        \item[(b)] $Q$ is the partition obtained from $\left(\mu,[n-1]^k\right)$, where $\mu=(\mu_1,\ldots, \mu_s)$ is a partition of $m+1$, and $\frac{n-1}{\mu_1}\leq k\leq n-1$.
        \item[(c)] $Q$ is the partition obtained from $\left(\mu, [n-2]^k\right)$, where either $\mu=(m+2)$ and $\frac{n-1}{m+1}< k \leq n-2$, or $\mu=(\mu_1, \mu_2, \ldots, \mu_s)$ is a partition of $m+2$ such that $\mu_2\geq 2$, and $\frac{n-1}{\mu_2} < k \leq n-2$.
    \end{enumerate}
\end{thm}

\begin{proof}
A generic matrix $A$ in $\U_B$ as in \eqref{generic in U} is written in the block form as in \eqref{blocks of generic in U}. We know that $H$ is similar to $J_n^k$ for some power $k\in\{1,2,\ldots,n\} $ and $V$ is similar to a nilpotent Jordan matrix with Jordan type $\nu=(\nu_1,\nu_2,\ldots,\nu_t)$. Then we have $H=S^{-1} J_n^k S $ and $V=T^{-1}J_{\nu}T $ for two invertible matrices $S$ and $T$. Since $H$ is strictly upper-triangular matrix we may choose $S$ such that $Se_1=\alpha e_1$ and $S^Te_n=\beta e_n$ for two nonzero scalars $\alpha$ and $\beta$. Then, we have
$$
A=\left( 
\begin{array}{c|c}
  S^{-1} J_n^kS   & e_1 u^T \\
  \hline
  d e_n^T   & T^{-1}J_{\nu}T
\end{array}
\right)=\left( 
\begin{array}{c|c}
  S^{-1}   & 0 \\
\hline
0   & T^{-1}
\end{array}
\right)\left( 
\begin{array}{c|c}
  J_n^k   & e_1 a^T \\
  \hline
  b e_n^T   & J_{\nu}
\end{array}
\right)\left( 
\begin{array}{c|c}
  S   & 0 \\
  \hline
  0   & T
\end{array}
\right),
$$
where we write $u=\alpha^{-1}T^Ta $ and $d=\beta T^{-1}b $.
Thus $A$ is similar to 
\[B=\left( 
\begin{array}{c|c}
  J_n^k   & e_1 a^T \\
  \hline
  b e_n^T   & J_{\nu}
\end{array}
\right)\]
and we will consider $B$ hereafter.

In order to find the Jordan type of $B$ we will compute its powers. Inductively, we prove that
$$
B^l=\left( 
\begin{array}{c|c}
  J_n^{kl} + \left(a^TJ_{\nu}^{l-2}b\right) J_n^{n-1}   & e_1 a^T J_{\nu}^{l-1} \\
  \hline
   J_{\nu}^{l-1} b e^T_n  & J_{\nu}^l
\end{array}
\right)
$$
for any $l\ge 2$. As long as $kl\leq n-2$, the off diagonal blocks do not affect the rank of $B^l$ and we have $\rank(B^l)=\rank(J_n^{kl})+\rank(J_{\nu}^l )$.

Consider next the case $kl=n-1$. We omit the first $n-1$ consecutive columns and $n-1$ consecutive rows, starting with the second row, in $B^l$ that are all zero to obtain
\begin{equation}
\label{kl=n-1}
\left( 
\begin{array}{c|c}
  1 + \left(a^TJ_{\nu}^{l-2}b\right)    & a^T J_{\nu}^{l-1} \\
  \hline
   J_{\nu}^{l-1} b   & J_{\nu}^l
\end{array}
\right).
\end{equation}
We consider several cases depending on the largest part $\nu_1$ of partition $\nu$. 

{\bf Case I.} $\nu_1\leq l-2$.

In this case we have $J_{\nu}^{l-2}=0$, and thus $\rank B^{l}=1$. We also have $\rank B^{s}=0$ for all $s>l$. The Jordan type of $B$ is equal to 
$\left(\nu, [n]^k\right)$ where $\nu$ is a partition of $m$ and $1\leq k\leq n$. So it belongs to the case (a) of Lemma \ref{commute with hook}.

{\bf Case II.} $\nu_1= l-1.$
Now, we have $J_{\nu}^{l-1} =J_{\nu} ^l=0$. So, $\rank B^l$ is either $0$ or $1$, depending on $1+a^TJ_{\nu}^{l-2}b$ being equal to $0$ or not, respectively. Suppose first that $\rank B^l=1$. Then, it follows that $\rank B^{s}=0$ for all $s>l$. Hence $\rank(B^l)=\rank(J_n^l)+\rank(J_{\nu}^l )$ for all $l$ and the Jordan type of $B$ is equal to 
$\left(\nu, [n]^k\right)$ where $\nu$ is a partition of $m$ and $1\leq k\leq n$. The Jordan type of $A$ in this subcase belongs to the case (a) of Lemma \ref{commute with hook}.

Suppose next that $\rank B^l=0$. Since $\nu_1=l-1$ and $kl=n-1$ it follows that the Jordan type of $B$ is equal to $\left(\nu', [n-1]^k\right)$, where $\nu'=(\nu_1+1,\nu_2,\nu_3,\ldots, \nu_t)=(l,\nu_2,\ldots,\nu_t)$ and $[n-1]^k=(l^k)$. Since $\nu_1+1=l=\frac{n-1}{k}$, it follows that the Jordan type in this case belongs to type (b) of Lemma \ref{commute with hook}.

{\bf Case III.} $\nu_1\ge l.$

Now the rank of $B^l$ can be equal to one of $3$ possible values: $\rank J_{\nu}^l$, $\rank J_{\nu}^l+1$ or $\rank J_{\nu}^l+2$, depending on $a$ and $b$. We consider the three cases separately.

\begin{itemize}
    \item[(i)] $\rank B^l-\rank J_{\nu}^l=0$: Then $1+a^TJ_{\nu}^{l-2}b=0 $, $J_{\nu}^{l-1}b\in\im J_{\nu}^l $ and $\left(J_{\nu}^{l-1}\right)^Ta\in\im \left(J_{\nu}^l\right)^T $. We have $\rank B^j=n-jk+\rank J_{\nu}^j $ for $j=0,1,\ldots,l-1$, and $\rank B^l=\rank J_{\nu}^l $. Since $n=kl+1$ it follows that the Jordan type of $B$ is equal to $([n-1]^k,\nu' ) $, where $\nu'$ is a partition of $m+1$ and $\nu_1'\ge\nu_1\ge l  = \frac{n-1}{k} $. 

So the Jordan type of $B$ belongs to type (b) of Lemma \ref{commute with hook}.

\item[(ii)] $\rank B^l-\rank J_{\nu}^l=1$: Then $\rank B^j=\rank J_n^{kj}+\rank J_{\nu}^j $ for all $j\leq l$. Since $J_n^{kl+1}=J_n^n=0$, the equality $\rank B^j=\rank J_n^{kj}+\rank J_{\nu}^j$ holds also for all $j> l$.  So, the Jordan type of $B$ is equal to $\left(\nu', [n]^k\right)$ where $\nu'$ is a partition of $m$ and $1\leq k\leq n$. It belongs to the case (a) of Lemma \ref{commute with hook}.

\item[(iii)] $\rank B^l-\rank J_{\nu}^l=2$: Then $\rank B^j=\rank J_n^{kj}+\rank J_{\nu}^j $ for all $j\leq l-1$ and $\rank B^l=\rank J_{\nu}^l+2=\rank J_n^{n-1}+\rank J_{\nu}^l+1$. Therefore, the Jordan type of $B$ is equal to $([n-1]^k,\nu')$, where $\nu'=(\nu'_1,\nu'_2,\ldots,\nu'_t)$ is a partition of $m+1$ and $\nu_1'\ge\nu_1+1\ge l+1\ge \frac{n-1}{k} $. It belongs to the case (b) of Lemma \ref{commute with hook}.

\end{itemize}

It remains to consider the case when there is no $l$ such that $kl=n-1$, i.e., we assume that the remainder of $n$ when divided by $k$ is not equal to $1$. We write $\ell=\left\lfloor{\frac{n}{k}}\right\rfloor$. Then 
$\rank B^j=\rank J_n^{kj}+\rank J_{\nu}^j $ for all $j\leq \ell$. 

\begin{itemize}
    \item[(i)] If $\nu_1\leq\ell$ then $B^j=0$ for $j\ge\ell +1$ and the Jordan type of $B$ is equal to $\left([n]^k,\nu\right)$ where $\nu$ is a partition of $m$ and $1\leq k\leq n$. It belongs to the case (a) of Lemma \ref{commute with hook}.

\item[(ii)] Assume next that $\nu_1\ge\ell+1$, and that $a^TJ_{\nu}^{\ell-1}b=0 $, $J_{\nu}^{\ell}b\in\im J_{\nu}^{\ell +1} $ and $\left(J_{\nu}^{\ell}\right)^Ta\in\im J_{\nu}^{\ell +1} $. Then, we have $\rank B^j=\rank J_n^{kj}+\rank J_{\nu}^j$  for all $j$ and the Jordan type of $B$ is equal to $\left([n]^k,\nu\right)$ where $\nu$ is a partition of $m$ and $1\leq k\leq n$. It belongs to the case (a) of Lemma \ref{commute with hook}.

\item[(iii)] Assume next that $\nu_1\ge\ell+1$, and that $a^TJ_{\nu}^{\ell-1}b\neq 0 $, but $J_{\nu}^{\ell}b\in\im J_{\nu}^{\ell+1}$ and $\left(J_{\nu}^{\ell}\right)^Ta\in\im J_{\nu}^{\ell+1}  $. Then $\rank B^j=\rank J_n^{kj}+\rank J_{\nu}^j $ for all $j\leq \ell$ and $\rank B^{\ell +1} =\rank J_{\nu}^{\ell +1} +1$. Therefore, the Jordan type of $B$ is equal to $([n-1]^k,\nu')$, where $\nu'=(\nu'_1,\nu'_2,\ldots,\nu'_t)$ is a partition of $m+1$ and $\nu_1'\ge\nu_1\ge \ell+1\ge \frac{n-1}{k} $. It belongs to the case (b) of Lemma \ref{commute with hook}.

\item[(iv)] If $\nu_1\ge\ell + 1$ but exactly one of the relations $J_{\nu}^{\ell}b\in\im J_{\nu}^{\ell +1}$ and $\left(J_{\nu}^{\ell}\right)^Ta\in\im J_{\nu}^{\ell +1} $ holds and the other fails, then $\rank B^j=\rank J_n^{kj}+\rank J_{\nu}^j $ for all $j\leq \ell$ and $\rank B^{\ell +1} =\rank J_{\nu}^{\ell +1} +1$. Therefore, the Jordan type of $B$ is equal to $([n-1]^k,\nu')$, where $\nu'=(\nu'_1,\nu'_2,\ldots,\nu'_t)$ is a partition of $m+1$ and $\nu_1'\ge\nu_1\ge \ell+1\ge \frac{n-1}{k} $. It belongs to the case (b) of Lemma \ref{commute with hook}.

\item[(v)] Finally, assume that $\nu_1\ge\ell +1$ and both relations $J_{\nu}^{\ell}b\notin\im J_{\nu}^{\ell +1}$ and $\left(J_{\nu}^{\ell}\right)^Ta\notin\im J_{\nu}^{\ell +1} $ hold. Then $\rank B^j=\rank J_n^{kj}+\rank J_{\nu}^j $ for all $j\leq \ell$ and $\rank B^{\ell +1} =\rank J_{\nu}^{\ell +1} +2$. Therefore, the Jordan type of $B$ is equal to $([n-2]^k,\nu')$, where $\nu'=(\nu_1,\nu_2,\ldots,\nu_s)$ is a partition of $m+2$. If $\nu=(m)$ then $s=1$ and $\nu'=(m+2)$. If $\nu_2\ge 1$ then $\nu'$ is a partition of $m+2$ such that $\nu_2'\ge \ell+1\ge 2$ and $\frac{n-1}{\nu'_2}\leq k \leq n-2$. This case belongs to the case (b) or (c) of Lemma \ref{commute with hook}.
\end{itemize}
\end{proof}

\section{Examples and consequences}\label{consequences}

In this section, we give some examples that illustrate our main result. We also discuss a few implications it has to the general question of which pairs of partitions are commuting Jordan types.

\begin{exmp} We choose $N=6$. Then there are three hook partitions with $n\ge 3$ and $m\ge 1$. These are $(5,1)$, $(4,1^2)$ and $(3,1^3)$. We list all partitions that commute with each of the three hook Jordan types except for the universally commuting partitions $(2^k,1^l)$ (see \cite[Thm. 2.4]{Obl-3} or \cite[Thm. 4.11]{BW}). In the table below, we group partitions commuting with a given hook partition  by their types described in parts (a), (b), and (c) of Lemma \ref{commute with hook}, respectively. Note that the same partition might occur in more than one column. In the last column, we list all the partitions that do not commute with a given hook partition.
\begin{center}
\begin{tabular}{|c||c|c|c|c|}
\hline
the hook  & type (a)   & type (b)   & type (c) & non-commuting \\
\hline \hline
$(5,1)$& $(5,1)$, $(3,2,1)$  &  &  $(3,1^3)$ & $(6)$, $(3,3)$ \\
\hline
$(4,1^2)$ & $(4,2)$, $(4,1^2)$  & $(3,3)$, $(3,2,1)$, &  & $(6)$, $(5,1)$ \\
& & $(3,1^3)$ & & \\
\hline
$(3,1^3)$ & $(3,3)$, $(3,2,1),$  &  $(4,2)$, $(4,1^2)$, & $(5,1)$ & $(6)$ \\
& $(3,1^3)$  & $(3,2,1)$ & & \\
\hline \hline
\end{tabular}
\end{center}

Observe that there are $11$ partitions of $N=6$. Of those, there are $4$ partitions that are universally commuting and $3$ Rogers-Ramanujan partitions. They form three pairs of non-commuting partitions. The one-part partition $(6)$ commutes only with $6$ almost rectangular partitions, so it forms non-commuting pairs with $5$ other partitions. It is not difficult to show that the only other pair of non-commuting partitions is $((3,3),(5,1))$ and that partitions $(4,2)$ and $(3,2,1)$ commute with $(3,3)$. (Here we refer to \cite[Thm. 3.1]{Obl-3} for the fact that $(3,3)$ and $(4,2)$ commute.) All together, there are $46$ commuting pairs of distinct partitions and $9$ non-commuting pairs of distinct partitions of $N=6$.

\end{exmp}

\begin{cor}
    A partition $P$ of $m+3$ commutes with $\Gamma_{3,m} =\left(3,1^m\right)$ if and only if the smallest part of $P$ is at most 3. Equivalently, the only partitions of $m+3$ that do not commute with $(3,1^m)$ are those with all parts greater than or equal to $4$.
\end{cor}
\begin{proof}
For $m=1$, the statement is clearly true as the only partition of 4 that does not commute with $(3,1)$ is $(4)$. 

Assume that $m\geq 2$, and let $P=(\lambda_1, \ldots, \lambda_r)$ be a partition of $m+3$ such that $\lambda_r\leq 3$. 
     
     If $\lambda_r=3$, then $(\lambda_1, \ldots, \lambda_{r-1})$ is a partition of $m$ and therefore, $P$ is a partition of the form given in part (a) of Lemma~\ref{commute with hook}.
     
     If $\lambda_r=2$, then $(\lambda_1, \ldots, \lambda_{r-1})$ is a partition of $m+1$. Since $\lambda_1\geq 2$, we have $\frac{3-1}{\lambda_1}\leq 1$, and therefore $P$ is a partition of the form described in part (b) of  Lemma~\ref{commute with hook}.
     
       If $\lambda_r=1$, then $(\lambda_1, \ldots, \lambda_{r-1})$ is a partition of $m+2$. If $r-1=1$, then $\lambda=(m+2,1)$. Since  $\frac{2}{m+1}<1$, in this case $P$ is a partition of the form described in part (c) of  Lemma~\ref{commute with hook}. Now assume that $r-1\geq 2$. In this case, if $\lambda_2> 2$, then $P$ is a partition of the form described in part (c) of  Lemma~\ref{commute with hook}, if $\lambda_2=2$, then $P$ is a concatenation of a partition of $m$ and $(2,1)=[3]^2$, which is a partition of the form described in part (a) of Lemma~\ref{commute with hook}. Finally, if $\lambda_2=1$, then $P=(\lambda_1, 1^{r-1})$ which is either $(m+1,1^2)$, when $r-1=2$, or a concatenation of a partition of $m$ and $[3]^3$. Thus, $P$ is a partition of the form described in part (b) or (a) of Lemma~\ref{commute with hook}.

 Finally, when $\lambda_r \geq 4$, by Theorem~\ref{all that commute with hook},  $P$ does not commute with $(3,1^m)$.
\end{proof}

The following example shows that partitions with the same generic commuting Jordan type do not necessarily commute themselves. In other words, there are noncommuting partitions $\Gamma$ and $P$ such that $\mathcal{D}(\Gamma)=\mathcal{D}(P)$.

\begin{exmp}\label{example.5}
    For a given Rogers-Ramanujan partition $Q=(p,q)$, $p\ge q+2$, there are two hook partitions in its inverse image under $\D$. These are $\Gamma_{p,q}=(p,1^q)$ and $\Gamma_{q+2,p-2} = (q+2,1^{p-2})$. 
    
First, we show that a hook partition of type $\Gamma_{p,q}$ does not commute with every other partition in the table $\D^{-1}(Q)$ in general.    

To begin, consider partitions $\Gamma=\Gamma_{8,2} =(8,1^2)$ and $P=(3^2,1^4)$. Generically, both $\Gamma$ and $P$ commute with $Q=(8,2)$. However, $P$ does not commute with $\Gamma$. This follows by Theorem \ref{all that commute with hook}. A partition commuting with $\Gamma$ has to contain an almost rectangular subpartition of either $8$, $7$ or $6$. Subpartition $(3^2)$ is the only such in $P$. It is now easy to check that the remaining subpartition $(1^4)$ of $P$ does not satisfy conditions for $\mu$ in part (c) of Theorem \ref{all that commute with hook}.

Next, consider all partitions in $\mathcal{D}^{-1}(13,3)$ (see Table 1.1 of \cite{IKvSZ} for a complete list). There are two hook partitions in $\mathcal{D}^{-1}(13,3)$, namely $\Gamma_{13,3}$ and $\Gamma_{5,11}$. While $\Gamma_{5,11}$ commutes with all other partitions in $\mathcal{D}^{-1}(13,3)$, the hook partition $\Gamma_{13,3}$ does not commute with the following set of partitions in $\mathcal{D}^{-1}(13,3)$:
    $$\left\{([7]^2,[9]^k)\mid 5\leq k \leq 9\right \}\cup \left\{([9]^3,[7]^k)\mid 6\leq k \leq 7\right\}.$$
This follows from Theorem~\ref{all that commute with hook}.

 Finally, we show that even a hook partition of the second type $\Gamma_{q+2,p-2}$ need not commute with every other partition in the table $\D^{-1}(Q)$ in general. The partition $\Gamma_{q+2,p-2}$ is a partition with $p-1$ parts which is the maximal number of parts of any partition in the table $\D^{-1}(Q)$.

    Consider partition $Q=(55,24)$. Then $\Gamma_{26,53}$ is a hook partition of the second type as described above. Partition $P=(10^3,7^7)$ is also a partition with $\D(P)=Q=(55,24)$. (These follows from \cite[Thm. 13]{Obl-1}.) Observe that by Theorem \ref{all that commute with hook} each partition that commutes with $\Gamma_{n,m}$ contains an almost rectangular subpartition of either $n$, $n-1$ or $n-2$. Partition $P$ has no almost rectangular subpartition of either $26$, $25$ or $24$. Therefore, it does not commute with $\Gamma_{26,53}$.
\end{exmp}

\bibliographystyle{plain}

\end{document}